\documentclass{article}
\usepackage[utf8x]{inputenc}
\usepackage{amssymb, amsmath, amsthm}
\theoremstyle{definition}

\newtheorem{theorem}{Theorem}[section]
\newtheorem{corollary}[theorem]{Corollary}

\newtheorem{proposition}[theorem]{Proposition}
\newtheorem{conjecture}[theorem]{Conjecture}
\title{Greedy Galois Games}
\author{Joshua Cooper \thanks{This work was funded in part by NSF grant DMS-1001370.} \\ Aaron Dutle \\ Department of Mathematics \\ University of South Carolina }
\date{\today}

\begin{document}

\maketitle

\begin{abstract} We show that two duelers with similar, lousy shooting skills (a.k.a.~{\em Galois duelers}) will choose to take turns firing in accordance with the famous Thue-Morse sequence if they greedily demand their chances to fire as soon as the other's {\em a priori} probability of winning exceeds their own.  This contrasts with a result from the approximation theory of complex functions that says what more patient duelers would do, if they {\em really} cared about being as fair as possible.  We note a consequent interpretation of the Thue-Morse sequence in terms of certain expansions in fractional bases close to, but greater than, $1$.
\end{abstract}

Two players, Alice and Bob, are in a duel.  They take turns firing at each other.  However, both are {\em Galois}\footnote{Famously, the prodigal algebraist and Republican Radical \'Evariste Galois lost a duel over a lover on May 30, 1832, dying the next day.} duelers, i.e., terrible shots, and equally so.  On the other hand, they are deeply committed to fairness, and therefore they make the following deal.  Before a single firearm is discharged, they draw up a firing sequence, i.e., the sequence of turns they will take, according to the following ``greedy'' rules.  Alice shoots first.  Bob then shoots as many times as he needs to obtain a probability of winning that meets or exceeds the probability that Alice has won so far.  Then Alice shoots again, until her {\em a priori} probability of having won exceeds Bob's.  Bob shoots next following the same rule, and so on until someone finally shuffles off his/her mortal coil.

To illustrate, suppose the duelers' hitting probability is $1/3$.  Alice shoots first, so her probability of winning by then end of round 0 is $1/3$.  Bob's probability of winning so far is zero, so he shoots next. For Bob to win in round 1, Alice has to have missed in round 0, and Bob has to hit. Therefore, Bob's probability of having won by the end of round 1 is $(2/3)(1/3) = 2/9$.  This is still less than $1/3$, so Bob shoots again in round 2. For Bob to win in round 2, he must survive Alice's initial shot, miss in round 1, and hit in round 2.  Hence his probability of winning by the end of round 2 is $(1/3)(2/3)+(1/3)(2/3)^2 = 10/27$.  This is more than Alice's probability of $1/3$ ($=9/27$), so Alice gets to go next.  In round three, Alice adds $(1/3)(2/3)^3$ to her probability of winning, since $1/3$ is the probability she succeeds in round 3, and $(2/3)^3$ is the probability that everyone missed in the previous 3 rounds.  If we define $S_{n, X} = \{i\leq n \mid \text{player } X \text{ shoots in round } i\}$, then the probability of player $X=A$(lice) or $B$(ob) winning in round $n$ is given by
$$
\frac{1}{3}\sum_{i\in S_{n,X}} \left(\frac{2}{3}\right)^i.
$$
The following is a table for $p=1/3,$ showing the probability of success for each player as well as the sequence of shooters.
$$
\begin{array}{c|cc|c}
\text{round}& \mathbb{P}(A) & \mathbb{P}(B) & \text{Shooter} \\
\hline
0& 1/3 & 0 & A \\
1& 1/3 & 2/9 & B \\
2& 1/3 & 10/27 & B \\
3& 35/81 & 10/27 & A \\
4& 35/81 & 106/243 & B \\
5& 347/729 & 106/243 & A \\
6& 347/729 & 1018/2187 & B \\
7& 347/729 & 3182/6561 & B \\
8& 9625/19683 & 3182/6561 & A \\
9& 9625/19683 & 29150/59049 & B \\
10& 87649/177147 & 29150/59049 & A
\end{array}
$$

For arbitrary probability $p$, we can determine the sequence $\{a_i\}_{i=0}^n$ of players inductively. Let $q=1-p$, let $a_n = -1$ mean that Alice shoots in round $n$, and $a_n = 1$ means that Bob shoots in round $n$. Let $A_n$ be the event that Alice wins by round $n$, and define $B_n$ similarly.  Since Alice shoots first, $a_0=-1$. Write
$$
f_n(q) = a_n\left(\sum_{j=0}^n a_jq^j\right).
$$
Then
\begin{align*} f_n(q) & =  a_n\left(\sum_{j=0}^n a_jq^j\right)\\
 & = a_n\left(\sum_{i \in  S_{n, B}} q^i - \sum_{i \in S_{n, A}}q^i\right) \\
 & = \frac{a_n}{p} \cdot \left( \mathbb{P}(B_n) - \mathbb{P}(A_n)\right)
\end{align*}
Since $a_n$ is negative whenever Alice is the shooter, we see that (up to the positive factor $1/p$) the polynomial $f_n(q)$ records the current player's probability of success minus the opposing player's probability of success.  Therefore, the next player is completely determined by the value of $f_n(q)$. Specifically,

\begin{equation} \label{coeff} a_{n+1} = \begin{cases} -a_n  & \text{if } f_n(q)\geq 0 \\
 a_n & \text{otherwise}. \end{cases} \end{equation}

It is easy to see that regardless of the value of $p$, the first three terms of the sequence $\{a_i\}$ are $-1$, $1$, $1$.  To determine the fourth term, we consider $f_2(q) = q^2+q-1$. The unique positive root of this polynomial is $\frac{-1+\sqrt{5}}{2} = \frac{1}{\phi} \approx 0.618$ where $\phi$ is the Golden ratio. Since $f_2(q)$ is increasing after this, we have that for any $q\geq 1/\phi,$ the fourth term of the sequence is $a_3= -1$.

The above is a special case of the following.
\begin{proposition} \label{converges}  For each $n\in \mathbb{N}$, there is an $\epsilon>0$ so that the sequence $\{a_i\}_{i=0}^n$ is the same for all $q\in (1-\epsilon, 1).$ \end{proposition}

\begin{proof} We proceed by induction, noting that the base case is trivial. Assume by induction that for all $q \in (1-\epsilon_0, 1),$ the sequence $\{a_i\}_{i=0}^n$ is the same.  Recall that $a_{n+1}$ is determined by the sign of $f_n(q)$, which is a now a fixed polynomial, since the coefficients are exactly the $a_i$.  Since $f_n$ has degree $n$, it has at most $n$ roots. Thus we can find $\epsilon_1 >0$ so that none of the roots occur in $(1-\epsilon_1, 1).$ Setting $\epsilon = \min \{\epsilon_0, \epsilon_1\},$ we have that $f_n(q)$ does not change sign or become zero for $q\in (1-\epsilon, 1)$. Therefore, $a_{n+1}$ does not depend on $q$ inside this interval, completing the induction, and proving the proposition.\end{proof}

One could continue along the lines above, and for each $n$, attempt to find the threshhold value of $q$ so that the first $n$ terms of the sequence stabilize. Indeed, the authors have done this for some small values of $n$, although none of the threshhold values other than $1/\phi$ appear to be numbers of independent interest.  However, we are willing to conjecture the following.

\begin{conjecture} Let $\{ \alpha_k \}_{k \geq 1}$ denote the sequence of those roots of the $f_n(q)$ lying strictly between $0$ and $1$.  Then $\alpha_k$ is increasing and $\alpha_k = 1 - \Theta(k^{-1/2})$.
\end{conjecture}

We are concerned here with a different question. As $q$ tends to 1 (i.e., $p \rightarrow 0$), what does the sequence of players tend to?  A quick calculation with $q=0.9$ reveals the following 21 turns in the sequence of players.
$$
ABBABAABBAABABBABAABB.
$$
At first glance, this appears to be the same as the famous (Prouhet-)Thue-Morse(-Euwe)\footnote{Prouhet used this sequence in 1851 to solve what is now known as the Prouhet-Tarry-Escott problem, although he did not make the sequence explicit.  Thue introduced it in 1906 to devise cube-free words, and Morse applied it to differential geometry in 1921.  Euwe, not knowing about these previous works, used the sequence in 1929 to show the existence of infinitely long chess games, despite the rule designed to prevent this: any three-fold repetition of a sequence of moves ends the game in a draw.  The reader is directed to the delightful survey \cite{Al98} for more of this sequence's interesting history.} sequence, one definition of which is the sequence of parities of the number of $1$'s in the binary expansions of $n$, $n = 1$, $2$, $\ldots$.  In fact, the sequence above differs only in the last position. This disagreement can be fixed by raising the value of $q$ very slightly (setting $q= 0.902$ is sufficient). That our sequence bears such close resemblance to the Thue-Morse sequence is no coincidence, as evidenced by the following.

\begin{theorem} \label{main} The sequence $\{a_i\}_{i=0}^\infty$ tends to the Thue-Morse sequence (on the alphabet $\{-1,1\}$) as $q \rightarrow 1^-$. \end{theorem}

Our proof will use the following well-known facts about the Thue-Morse sequence, which can found, for example, in \cite{Lo02}.

\begin{proposition} \label{TMrel} The Thue-Morse sequence $\{t_i\}_{i=0}^\infty $ on alphabet $\{-1,1\}$ is defined by the following recurrences.
\begin{align*} t_0 & = -1\\
t_{2i} &= t_i \\
t_{2i+1} &= (-1)t_{2i}.\end{align*} \end{proposition}

\begin{proposition} \label{block} The sequence $\{ (t_{2i}, t_{2i+1})\}_{i=0}^\infty$ is the Thue-Morse sequence on alphabet $\{(-1,1), (1,-1)\}.$ \end{proposition}

We note a simple consequence of Proposition \ref{block} which we will also use.

\begin{corollary} \label{balanced} For any $n\in \mathbb{N},$ we have $\sum_{i=0}^{2n+1} t_i = 0.$ \end{corollary}

\begin{proof}[Proof of Theorem \ref{main}]
In light of Proposition \ref{converges}, $q$ can be taken arbitrarily close to 1.  We proceed by induction. We have already shown that the two sequences agree for $i = 0$, $1$, $2$, $3$, so the base cases hold. We assume $n>2$, and by induction that the two sequences agree for all $i\leq n.$\\

\noindent \emph{Case 1:} $n=2m$ is even.
Consider $g(q) = \sum_{i=0}^{n-1} a_iq^i.$ Since the $a_i$ are the Thue-Morse sequence, Corollary \ref{balanced} tells us that $g(1)=0.$ Since $q$ can be taken arbitrarily close to $1$ and $g$ is continuous, we can ensure $-1/2<g(q)<1/2$ for all $q$ under consideration. We may also assume that $q>(1/2)^{1/n}$.  Then note that $f_n(q) = q^n \pm g(q),$ so that for all of our $q$,
$$
f_n(q) = q^n \pm g(q) > 1/2 - 1/2 \geq 0
$$
Thus (\ref{coeff}) gives that $a_{n+1} = (-1)a_{n}.$ Since $n$ is even, induction and the recurrence for Thue-Morse give that $$a_{n+1} = (-1)a_n =  (-1)t_n = t_{n+1}.$$

\noindent \emph{Case 2:} $n=2m+1$ is odd.
Since $n$ is odd, Corollary \ref{balanced} gives that $f_n(1) = 0.$ Hence we can write
$$f_n(q) = (q-1)g(q)$$ for some monic degree $2m$ polynomial $g$.

We claim that $g(q) = f_m(q^2)$. We know by induction that sequence $\{a_i\}$ matches Thue-morse up to $n$,  whence $a_{2i+1} = (-1)a_{2i}$ and $a_{2i} = a_i$ for all of the coefficients in our polynomial.  We can write
\begin{align*} f_n(q) & = a_{2m+1} \left(\sum_{i=0}^{2m+1} a_iq^i\right)\\
& = a_{2m+1} \left(\sum_{i=0}^m a_{2i}q^{2i}+a_{2i+1}q^{2i+1}\right) \\
& = (-1)a_{2m} \left(\sum_{i=0}^m a_{2i}q^{2i}-a_{2i}q^{2i+1}\right) \\
& = (-1)a_{2m} (1-q)\left(\sum_{i=0}^m a_{2i}q^{2i}\right) \\
& = (q-1)a_m\left(\sum_{i=0}^m a_{i}(q^2)^{i}\right) \\
& = (q-1)f_m(q^2), \end{align*} proving the claim.

Now, note that for $q$ in our range, $(q-1)$ is negative, so  one of $f_n(q)$ and $f_m(q)$ is positive, and the other is negative. Therefore, (\ref{coeff}) says that  for some $j\in \{0,1\}$, $a_{n+1} = (-1)^ja_n$ and $a_{m+1} = (-1)^{j+1}a_m.$ Then since we know the Thue-Morse relations hold up to $n$, we have
$$
a_{n+1} = (-1)^ja_n = (-1)^ja_{2m+1} = (-1)^{j+1}a_{2m} = (-1)^{j+1}a_{m} = a_{m+1}.
$$

By the inductive hypothesis and Proposition \ref{TMrel}, $a_{m+1} = t_{m+1} = t_{2m+2} = t_{n+1}$, completing the proof.
\end{proof}

Not all of this gun violence is fun and games.  Indeed, it bears on some serious business in the approximation theory of complex functions.  G\"unt\"urk showed (\cite{Gu05}) that, in a sense, there is an even ``fairer'' sequence than the Thue-Morse sequence, if only the shooters were not so greedy.  Indeed, his much more general result is the following.

\begin{theorem} Let $0 \leq \mu < 1 \leq M < \infty$ be arbitrary and $\mathcal{R}_M = \{z \in \mathbb{C} : |1-z| \leq M(1-|z|)\}$.  There exist constants $C_1 = C_1(\mu,M) > 0$ and $C_2 = C_2(\mu,M) > 0$ such that, for any power series
$$
f(z) = \sum_{n=0}^\infty a_n z^n, \qquad a_n  \in [-\mu,\mu], \; \forall n,
$$
there exists a power series with $\pm 1$ coefficients, i.e.,
$$
Q(z) = \sum_{n=0}^\infty b_n z^n, \qquad b_n \in \{-1,+1\}, \; \forall n,
$$
which satisfies
$$
|f(z) - Q(z)| < C_1 e^{C_2 / |1-z|}
$$
for all $z \in \mathcal{R}_m \setminus \{1\}$.
\end{theorem}

Furthermore, this result is best possible in a certain precise sense by a theorem of Borwein-Erd\'elyi-K\'os (\cite{Bo99}).  If one sets $\mu = 0$ and $M=1$, we obtain the corollary that one can approximate the constant $0$ function (of $p$) within $\exp(-c/p)$ for some constant $c$ by a power series with coefficients in $\{-1,+1\}$; the author goes on to show that the Thue-Morse sequence only obtains an approximation of $\exp(-c (\log p)^2)$.  It pays to have patience!

In fact, G\"unt\"urk communicates an observation by Konyagin, who asked the approximation question to begin with: if one takes as $\mathbf{a} = \{a_n\}_{n \geq 0}$ any expansion of $\frac{1}{2p}$ in the fractional base $\frac{1}{q}$, the generating function $g(z)$ of $\mathbf{a}$ satisfies $g(p) = 0$, for any $q \in (0,1/2)$.  An {\em expansion of $x \in \mathbb{R}^+$ in the base $\beta \geq 1$} is any sequence of the form
$$
c_n c_{n-1} \ldots c_1 c_0 . c_{-1} c_{-2} \ldots,
$$
where, for each $k \leq n$, $c_n \in \{0,1,\ldots,\lfloor \beta \rfloor\}$ and
$$
x = \sum_{k=0}^\infty c_{n-k} \beta^{n-k}.
$$
Such expansions were introduced by R\'{e}nyi under the name ``$\beta$-expansions'' (\cite{Re57}).  We may therefore reinterpret Theorem \ref{main} as follows.

\begin{corollary} \label{cor:expansion} Any initial segment of the Thue-Morse sequence, expressed over the alphabet $\{0,1\}$ -- for either choice of assignment of $\{\text{Alice},\text{Bob}\}$ to (distinct elements of) $\{0,1\}$ -- agrees with some expansion of $n/2$ in the base $1 + 1/n$ all of whose digits lie to the right of the radix point, for any sufficiently large $n$.
\end{corollary}

\begin{proof} The quantities $g_n(q) = \frac{1}{p} \sum_{i=0}^{n-1} a_i q^i$ and $\bar{g}_n(q) = \frac{1}{p} \sum_{i=0}^{n-1} \overline{a_i} q^i$ (where $\bar{\cdot}$ denotes boolean negation) represent, respectively, Bob's and Alice's probabilities of having won after $n$ rounds.  Since $g_n(q)$ and $\bar{g}_n(q)$ switch order (i.e., the larger becomes smaller and the smaller larger) infinitely often as $n$ increases, are both monotone nondecreasing in $n$, and
$$
p (g(q) + \bar{g}(q)) = p \sum_{i=0}^{n-1} q^i = p \cdot \frac{1 - q^n}{1 - q} = 1 - q^n,
$$
which tends to $1$ as $n \rightarrow \infty$, we may conclude that
$$
\lim_{n \rightarrow \infty} g_n(q) = \lim_{n \rightarrow \infty} \bar{g}_n(q) = \frac{1}{2p}.
$$
Setting $p = 1/n$ and $q = 1 - 1/n$ gives the stated result, since the sequence of coefficients of $g_n$ and of $\bar{g}_n$ tends to a limit as $n \rightarrow \infty$.
\end{proof}

R\'enyi referred to a greedily-constructed $\beta$-expansion as ``the'' $\beta$-expansion.  It is not difficult to see now that ``the'' $\beta = 1 + 1/n$ expansion of $n/2$ is exactly Alice and Bob's firing sequence, where Alice is associated with $1$ and Bob with $0$.  For example, the $(3/2)$-expansion of $1$ is
$$
0.10010100101\ldots
$$
This example also shows how $\beta$-expansions need not be unique if $\beta \not \in \mathbb{N}$.  By Corollary \ref{cor:expansion}, switching $1$ with $0$ gives another representation of $n/2$ in the same base!

We conclude with two questions.  What happens in a Galois {\em truel}, i.e., a three-way duel between equally terrible shots who are nonetheless fair-minded, optimally strategic, and are unwilling to deliberately miss?  It is not immediately clear what the fairest policy for turn-taking should be.  Finally, can Alice and Bob make their game fairer by imposing `less greedy' demands on the turn sequence which are still bounded computations run on the sequence so far?

\end{document}